\documentclass[reqno,a4paper]{amsart}
\usepackage[english,activeacute]{babel}
\usepackage{amssymb,amsmath,amsthm,amsfonts,mathrsfs,scalefnt,graphicx,graphics,color,hyperref}
\usepackage{lmodern}	
\usepackage{textcomp}
\usepackage{tikz}
\usepackage{listings}
\usetikzlibrary{trees}
\usetikzlibrary{shapes}
\usepackage[on]{auto-pst-pdf}
\usepackage{pst-tree}
\usepackage{pst-all}
\usepackage{tabularx}
\usepackage[T1]{fontenc}
\usepackage{url}
\usepackage[latin1]{inputenc}
\theoremstyle{plain}
\newtheorem{theorem}{Theorem}[section]
\newtheorem{lemma}[theorem]{Lemma}
\newtheorem{proposition}[theorem]{Proposition}
\newtheorem{corollary}[theorem]{Corollary}
\newtheorem{definition}[theorem]{Definition}
\newtheorem{hyp}[theorem]{Assumption}
\newtheorem{remark}[theorem]{Remark}
\numberwithin{equation}{section}
\newcommand{\minibox}[2][c]{\begin{tabular}{#1}#2\end{tabular}}

\newcommand{\Rb}  {{\mathbb R}}

\newcommand{\As} {{\mathcal A}}

\newcommand{\Fs} {{\mathcal F}}

\newcommand{\Ms} {{\mathcal M}}

\newcommand{\Ps} {{\mathcal P}}
\newcommand{\Qs} {{\mathcal Q}}
\newcommand{\Ss} {{\mathcal S}}

\newcommand{\al}{\alpha}

\renewcommand{\phi}{\varphi}

\newcommand{\la}{\lambda}

\newcommand{\ind}{1\!\kern-1pt \mathrm{I}}
\newcommand{\rsto}{]\!\kern-1.8pt ]}
\newcommand{\lsto}{[\!\kern-1.7pt [}

\newcommand\F{\mbox{I\kern-2pt F}}

\newcommand\cQ{{\mathcal Q}}
\newcommand\cS{{\mathcal S}}

\title[Binary markets under transaction costs]{Binary markets under transaction costs}
\author{Fernando Cordero}
\address{Faculty of Mathematics, University of Vienna, Nordbergstrasse 15, 1090 Vienna, Austria.}
\email{cordiery@gmail.com, fernando.cordero@univie.ac.at}

\author{Irene Klein}
\address{Department of Statistics and Operations Research, University of Vienna, Br\"{u}nnerstrasse 72, 1210 Vienna, Austria.}
\email{irene.klein@univie.ac.at}

\author{Lavinia Ostafe}
\address{Faculty of Mathematics, University of Vienna, Nordbergstrasse 15, 1090 Vienna, Austria.}
\email{lavinia.ostafe@univie.ac.at}
\thanks{The first author gratefully acknowledges financial support from the European Research Council (ERC) under grant agreement No.~247033. The third author gratefully acknowledges financial support from the Austrian 
Science Fund (FWF) under grant P19456 and from the European Research 
Council (ERC) under grant No.~247033.}

\date{\today}%

\begin{document}

\begin{abstract}
The goal of this work is to study  binary market models with transaction costs, and to characterize their arbitrage opportunities. It
has been already shown that the absence of arbitrage is related to the existence of $\la$--consistent price systems ($\lambda$-CPS), and, for this reason,
we aim to provide conditions under which such systems exist. More precisely, we give a characterization for the smallest transaction cost $\la_c$ (called ``critical'' $\la$)
starting from which one can construct a $\la$--consistent price system. We also provide an expression for the set $\Ms(\lambda)$ of all probability measures
inducing $\lambda$-CPS. We show in particular that in the transition phase ``$\lambda=\lambda_c$'' these sets are empty if and only  if the frictionless market
admits arbitrage opportunities. As an application, we obtain an explicit formula for $\lambda_c$ depending only
on the parameters of the model for homogeneous and also for some semi-homogeneous binary markets.
\end{abstract}

\maketitle
\section{Introduction}\label{s0}
Intuitively, a binary market is a market in which the stock price process $(S_n)_{n=0}^N$ is an adapted stochastic process with strictly positive values and such
that at time $n$ the stock price evolves from $S_{n-1}$ to either $\alpha_n\, S_{n-1}$ or $\beta_n\, S_{n-1} $, where $\beta_n <\alpha_n$. The values $\alpha_n$ and $\beta_n$
depend only on the past. So there are exactly $2^n$ different possible paths for the stock price to evolve up to time $n$.

The study of binary market models is both interesting and useful in order to obtain more information about the behavior of 
continuous models. This is indeed the case, as a typical situation that may occur is when a continuous model can
be expressed as a limiting process of a sequence of binary market models. Such a construction comes very natural for the Black-Scholes models 
which are driven by a standard Brownian motion. The key point is to approximate, by means of the Donsker theorem, the Brownian motion by a 
random walk consisting of independent Bernoulli random variables with the same parameter.

Moreover, this idea can be extended also to Black-Scholes-type markets that are driven by a process, for which
we dispose of a random walk approximation. Examples of this are the fractional Brownian motion and the Rosenblatt process, as one can see in \cite{Sotti} and \cite{Totu} respectively. 
In these works, the authors construct a sequence of binary models approximating the fractional Black--Scholes (respectively the Rosenblatt Black--Scholes) by giving an 
analogue of the Donsker's theorem, which, in this case, means that the fractional Brownian motion (respectively the Rosenblatt process) can be 
approximated by a ``disturbed'' random walk.

An important feature for a binary market model that one can study is its arbitrage opportunities. In \cite{Dzh} 
Dzhaparidze extensively describes a general mathematical model for the finite binary securities market in which he 
gives a complete characterization of the absence of arbitrage by using ideas of Harrison and Pliska \cite{Hapl}. However, interesting binary 
models admitting arbitrage opportunities can be found in the literature. Indeed, in \cite{Sotti}
 Sottinen showed that the arbitrage persists in the fractional binary markets approximating the fractional Black--Scholes 
 and such an opportunity is explicitly constructed using the path information starting from time zero. An analogous result for Rosenblatt binary markets
 is obtained in \cite{Totu}.
 
 All the above--mentioned results were obtained for a binary market model without transaction costs. 
In the present paper we focus our attention on the study of binary market models under transaction costs 
$\la$ and their arbitrage opportunities. When one introduces transaction 
costs, the usual notion of an equivalent martingale measure that is used in a market without friction is replaced 
by the concept of a $\la$--consistent price system ($\la$--CPS). In this work, we aim to give necessary and sufficient 
condition under which a binary market is ``good'' or not. By ``good'' 
we mean that the parameters of the model are given in such a way that there exist consistent price systems. 
Notice that this is not always the case, as one could see from the above discussion. Therefore, we
characterize the smallest transaction costs $\la_c$ starting from which one can construct a $\la$--CPS. 
By using this characterization, we can obtain an explicit expression for $\la_c$ when the parameters of the model 
are homogeneous in time and space, but also for a large class of 
semi--homogeneous cases, i.e.~when the parameters of the model are not necessarily 
homogeneous in time but they are still homogeneous in space.

The paper is organized as follows. In Section~2, we start introducing some notations and definitions concerning binary markets that we will use along this work. We recall
necessary and sufficient conditions to exclude arbitrage opportunities for these markets in the frictionless case (see \cite{Dzh}). Finally,
we present the notion of $\lambda$-consistent price system and we state the Fundamental Theorem of Asset Pricing which permits to relate the
existence of consistent price systems to the absence of arbitrage opportunities.

In Section~3, we give a brief presentation of the 1-step model, in which all the calculations are explicit. We also show that the results for 1-step models
allow to obtain a lower bound  in the general case.

Section 4 consists of technical lemmas which are used to establish necessary and sufficient conditions for the existence of $\lambda$-CPS.

In Section 5 the main results are concentrated. We start with a characterization for the smallest transaction cost $\la_c$ (called ``critical'' $\la$)
starting from which one can construct a $\la$--consistent price system. In a similar way, we obtain an expression for the set $\Ms(\lambda)$ of all probability measures
inducing $\lambda$-CPS. These results are a consequence of the necessary and sufficient conditions established in Section 4. We finish this section proving that
a binary market with critical transaction costs $\lambda_c$ admits arbitrage if and only if the corresponding frictionless market admits arbitrage.

In Section 6, we apply our results to give an explicit formula for the critical transaction costs $\lambda_c$ for homogeneous and
some semi-homogeneous binary markets.

Even if the binary models in the setting of no transaction costs were already studied in the literature, 
there is no result which gives us the conditions under which there exist 
consistent price systems when one passes to the case of transaction costs. This is precisely the goal of this 
paper.

\section{Preliminaries}\label{s1}

\subsection{Definitions}
To formalize the notion of a binary market, we introduce first some notations which will be useful along to this work. For a detailed treatment of this subject see \cite{Dzh} and Section II.1e of \cite{Shir}.

\subsubsection{The market model}
Let $(\Omega,\Fs,{(\Fs_n)}_{n=0}^N, P)$ be a finite filtered probability space. By a binary market we mean a market in which two
assets (a bond $B$ and a stock $S$) are traded at successive times $t_0=0<t_1<\cdots<t_N$. The evolution of the bond and stock is
described by:
$$B_n=(1+r_n)B_{n-1}$$
and
\begin{equation}\label{stock}
 S_n= \left(a_n+(1+X_n)\right)\,S_{n-1},\quad \forall n\in\{1,...,N\},
\end{equation}
where $r_n$ and $a_n$ are the interest rate and the drift of the stock in the time interval $[t_n,t_{n+1})$. The value of $S$ at time
$0$ is given by:
$$S_0=s_0=1+a_0+x_0.$$
We may assume, for the sake of simplicity, that the bond plays the role of a num\'eraire, and, in this case, that it
is equal to $1$ at every time $n$.
The process $(X_n)_{n=0}^N$ is an adapted stochastic process starting at
$X_0=x_0$ and such that, at each time $n$, $X_n$ can take only two possible values $u_n$ and $d_n$ with $d_n<u_n$. While $a_n$ from \eqref{stock}
is deterministic, the values of $u_n$ and $d_n$ may depend on the path of $X$ up to time $n-1$. This means that if, for each $n>1$, we denote by
$\vec{X}_{n-1}=(X_{n-1},...,X_0)$ and by
$$E_{n-1}=\{\vec{X}_{n-1}(\omega)\ :\ \omega\in\Omega\}$$
the set of all possible paths up to time $n-1$, then $X_n\in\{u_n(\vec{X}_{n-1}),d_n(\vec{X}_{n-1})\}$. For $n=1$, we have that $E_0=\{x_0\}$. \\
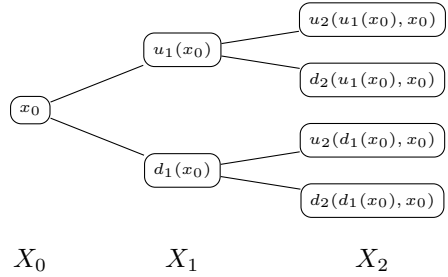
\begin{center}
\begin{tikzpicture}[grow=right]
\tikzstyle{level 0}=[rectangle,rounded corners, draw,level distance=10mm]
\tikzstyle{level 1}=[rectangle,rounded corners, draw,level distance=20mm, sibling distance=16mm]
\tikzstyle{level 2}=[rectangle,rounded corners, draw,level distance=25mm, sibling distance=8mm]
\node[level 0] {\tiny{$x_0$}}
child{node[level 1]{\tiny{$d_1(x_0)$}}
  child{node[level 2]{\tiny{$d_2(d_1(x_0),x_0)$}}}
  child{node[level 2]{\tiny{$u_2(d_1(x_0),x_0)$}}}
}
child{node[level 1]{\tiny{$u_1(x_0)$}}
  child{node[level 2]{\tiny{$d_2(u_1(x_0),x_0)$}}}
  child{node[level 2]{\tiny{$u_2(u_1(x_0),x_0)$}}}
};
\node at (0,-2){\minibox{$X_0$}};
\node at (2,-2) {\minibox{$X_1$}};
\node at (4.5,-2) {\minibox{$X_2$}};
\end{tikzpicture}
\end{center}
Now, we put for each $y\in E_{n-1}$:
$$\alpha_n(y)= 1+a_n+u_n(y)\quad\textrm{and}\quad\beta_n(y)=1+a_n+d_n(y),$$
and we assume that $\alpha_n(y)$ and $\beta_n(y)$ are strictly positive for every $y\in E_{n-1}$.\\
\begin{hyp}\label{ass}
We assume in addition that:
\begin{itemize}
\item The filtration ${(\Fs_n)}_{n=0}^N$ coincides with the natural filtration of $(X_n)_{n=0}^N$.
\item For all $\omega\in\Omega$, $\{\omega\}\in\Fs_N$.
\item  For all $\omega\in\Omega$, $P(\{\omega\})>0$.
\end{itemize}
\end{hyp}
\begin{remark}
The first two conditions of Assumptions~\ref{ass} allow to identify the spaces $\Omega$ and $E_N$ as well as the spaces of probability measures $\Ps_1(\Omega)$ and $\Ps_1(E_N)$.
We can also identify $\Ps_1(\Omega)$ with:
$$ {[0,1]}_*^{2^N-1} =\left\{\left(q_n(\Qs,x):1\leq n\leq N,\, x\in E_{n-1}\right): q_n(x)\in[0,1]\right\}$$
by means of the relation:
\begin{equation}\label{qn}
q_n(\Qs,y):=\Qs\left(X_n=u_n(y)\Big|\,\vec{X}_{n-1}=y\right)
\end{equation}
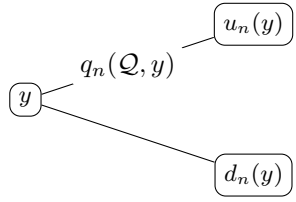
\begin{center}
\begin{tikzpicture}[grow=right]
\tikzstyle{level 0}=[rectangle,rounded corners, draw,level distance=20mm]
\tikzstyle{level 1}=[rectangle,rounded corners, draw,level distance=30mm, sibling distance=20mm]
\node[level 0] {\small{$y$}}
child{node[level 1]{\small{$d_n(y)$}}}
child{node[level 1]{\small{$u_n(y)$}}
edge from parent node[fill=white] {$q_n(\Qs,y)$}};
\end{tikzpicture}
\end{center}
\vspace{.2cm}
When there is no risk of confusion we write $q_n(y)$ instead of $q_n(\Qs,y)$.\\
We are using the notation ${[\ ,\ ]}_*$ to emphasize that the coordinates of a vector will be associated to nodes in the tree. Thus, for example, when we speak
of continuity of a function in $\Ps_1(\Omega)$, we refer to the continuity of the function viewed as a function in ${[0,1]}_*^{2^N-1}$, so coordinate by coordinate.
More precisely, we can define the metric $d_\infty$ in $\Ps_1(\Omega)$ as:
$$d_\infty(\Qs,\widehat{\Qs})=\max_{n\in\{1,...,N\}}\left\{\max_{x\in E_{n-1}}|q_n(\Qs,x)-q_n(\widehat{\Qs},x)|\right\}.$$
\end{remark}
\begin{remark}\label{eqprob}
The last condition of Assumption~\ref{ass} implies that:
$$\Qs\sim P \Longleftrightarrow\Qs(\{\omega\})>0,\textrm{ for all }\omega\in\Omega.$$
\end{remark}

\subsubsection{Notations on the binary tree}
In order to simplify the notations, we introduce the operators extension ``$\star u$'' and ``$\star d$'' acting on the nodes of the binary tree of the paths of the process $X$.

For $n\in\{1,...,N\}$, $y=(y_{n-1},...,y_0)\in E_{n-1}$, we define:

\begin{itemize}
\item[]$y\star u^0=y\quad\textrm{and}\quad y\star d^0=y.$
\item[] $y\star u=(u_n(y),y)\quad\textrm{and}\quad y\star d=(d_n(y),y).$
\item[] $y\star u^{i+1}=(y\star u^i)\star u\quad\textrm{and}\quad y\star d^{i+1}=(y\star d^i)\star d\,;\quad$ for $i\in\{0,..., N-n\}$.
\end{itemize}

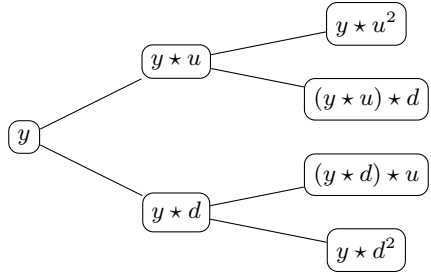
\vspace{.3cm}
\begin{center}
\begin{tikzpicture}[grow=right]
\tikzstyle{level 0}=[rectangle,rounded corners, draw,level distance=10mm]
\tikzstyle{level 1}=[rectangle,rounded corners, draw,level distance=20mm, sibling distance=20mm]
\tikzstyle{level 2}=[rectangle,rounded corners, draw,level distance=25mm, sibling distance=10mm]
\node[level 0] {\small{$y$}}
child{node[level 1]{\small{$y\star d$}}
  child{node[level 2]{\small{$y\star d^2$}}}
  child{node[level 2]{\small{$(y\star d)\star u$}}}
}
child{node[level 1]{\small{$y\star u$}}
  child{node[level 2]{\small{$(y\star u)\star d$}}}
  child{node[level 2]{\small{$y\star u^2$}}}
};
\end{tikzpicture}

\end{center}

\subsection{No arbitrage condition in the frictionless case}\label{ssnac}
We know by Proposition 3.6.2 in \cite{Dzh} that a binary market excludes arbitrage opportunities if and only if for all $n\in\{1,...,N\}$
and $y\in E_{n-1}$, we have:
\begin{equation}\label{nac1}
d_n(y)<-a_n< u_n(y),
\end{equation}
or equivalently:
\begin{equation}\label{nac2}
\beta_n(y)<1< \alpha_n(y).
\end{equation}
This is related to the existence of a probability measure $\Qs^0$ equivalent to $P$ such that
${(S_n)}_{n=0}^N$ is a $\Qs^0$-martingale. It is easy to see that such $\Qs^0$ must satisfy for each $n\in\{1,...,N\}$ and $x\in E_{n-1}$:
\begin{equation}\label{eqme}
\Qs^0\left(X_n=u_n\left(\vec{X}_{n-1}\right)\Big\arrowvert\vec{X}_{n-1}=x\right)= \frac{-a_n-d_n(x)}{u_n(x)-d_n(x)}=\frac{1-\beta_n(x)}{\alpha_n(x)-\beta_n(x)}.
\end{equation}
Moreover, under condition \eqref{nac1}, identity \eqref{eqme} defines a unique equivalent martingale measure. See Chapter 3 of \cite{Dzh} for more details.
\subsection{\texorpdfstring{Transaction costs and $\lambda$-CPS}{}}
Now, we introduce proportional transaction costs $\lambda\in]0,1[$ in our binary market $S$, which means that the bid and ask price of the stock $S$ are modeled by the processes
${((1-\lambda)S_n)}_{n=0}^N$ and ${(S_n)}_{n=0}^N$ respectively. In this framework the notion of consistent price system
replaces the notion of equivalent martingale measure that is used in a market without transaction costs, and, one can relate the absence of
arbitrage to the existence of such systems.
\begin{definition}[$\lambda$-consistent price system]
 A $\lambda$-consistent price system ($\lambda$-CPS) for the binary market $S$ is a pair $(\Qs,\tilde{S})$ of a probability measure
$\Qs\sim P$ and a process ${(\tilde{S}_n)}_{n=0}^N$ which is a martingale under $\Qs$ such that:
\begin{equation}\label{CPSd}
(1-\lambda)S_n\leq \tilde{S}_n\leq S_n,\qquad \textrm{a.s., for all $n\in\{0,...,N\}$}.
\end{equation}
We denote by $\cS^{\la}$ the set of $\lambda$-CPS.
\end{definition}

The following Theorem relates the existence of consistent price systems to the absence of arbitrage. A proof for it can be found, for
example, in \cite{Sch}.
\begin{theorem}[Fundamental Theorem of Asset Pricing in the case of finite $\Omega$]\label{TFAP}
Given a stock price process $S=(S_n)_{n=0}^N$ on a finite probability space and transaction costs $0<\la<1$, the following are equivalent:
\begin{enumerate}
 \item The process $S$ does not allow for an arbitrage under transaction costs $\la$.
 \item $\cS^{\la}\neq \emptyset$.
\end{enumerate}
\end{theorem}
Now, define $\Ms(\lambda)$, the set of all probability measures $\Qs\sim P$ inducing a $\lambda$-CPS, that is:
$$\Ms(\lambda)=\left\{ \Qs\sim P:\,\exists\, \tilde{S}\textrm{ such that }(\Qs,\tilde{S}) \textrm{ is a $\lambda$-CPS}\right\}.$$
One of the goals of this work is to characterize these sets. The other goal is to characterize the critical transaction costs $\lambda_c$, starting from which
the arbitrage opportunities disappear. Using Theorem~\ref{TFAP}, we can express $\lambda_c$ as:
$$\lambda_c=\inf\{\lambda\in[0,1]: \textrm{s.t. }\exists\ \lambda-\textrm{CPS for}\,\, (S,P)\}.$$
By definition, when we assume that the binary market model with $0$ transaction costs excludes arbitrage opportunities, then $\la_c=0$.
\section{\texorpdfstring{The 1-step model and a general lower bound for $\lambda_c$}{}}
We start this paragraph by analyzing the $1$-step model ($N=1$) in which we can explicitly find an easy expression for $\la_c$.
Indeed, if we assume that for some $\la$ there exists a $\la$--CPS $(\Qs,\widetilde{S})$, then, by the martingale property of $\widetilde{S}$ and the
inequality \eqref{CPSd}, we obtain that
\begin{equation}\label{ineq}
0\vee\left(\frac{1-\la-\beta_1(x_0)}{\al_1(x_0)-\beta_1(x_0)}\right)\leq\Qs(X_1=u_1(x_0))\leq 1\wedge\left(\frac{\frac1{1-\la}-\beta_1(x_0)}{\al_1(x_0)-\beta_1(x_0)}\right).
\end{equation}
Using Remark~\ref{eqprob}, it directly follows that
$$\lambda>1-\alpha_1(x_0)\ \ \mathrm{and}\ \ \lambda>1-\frac{1}{\beta_1(x_0)}.$$
In the other direction, if we start with some transaction costs $\la$ given as above, then we can choose a probability
measure $\cQ$ satisfying \eqref{ineq}, and, hence, find a process $\widetilde{S}$ which is a $\cQ$--martingale and satisfies \eqref{CPSd}.
The argument and a more detailed presentation of this example can be found in \cite{Sch}.
We therefore have that:
\begin{equation}\label{r1s}
\lambda_c^{(1)}=1-\alpha_1(x_0)\wedge\frac{1}{\beta_1(x_0)}\wedge 1.
\end{equation}
However, we will see in the next sections that it is more complicated to construct a $\lambda$-CPS for a general $N$-step model than to
construct $\lambda$-CPS for each 1-step sub-binary market and then to paste them together. Even so, this naive idea permits to give a lower bound
for the critical transaction costs $\lambda_c$.
\begin{proposition}[Lower bound for $\lambda_c$]\label{plblc} We have that:
$$\lambda_c\geq\lambda_*=1-\min\limits_{n\in\{1,...,N\}}\left\{\min\limits_{x\in E_{n-1}}\left\{\alpha_n(x)\wedge\frac{1}{\beta_n(x)}\wedge 1\right\}\right\}.$$
\end{proposition}
\begin{proof}
If we take $\lambda>\lambda_c$, then there exists a $\lambda$-CPS $(\Qs,\tilde{S})$. We divided the $N$-step binary market in $2^N-1$ $1$-step binary markets.The restriction of $(\Qs,\tilde{S})$ to each
one of this binary markets is also a $\lambda$-CPS. By using the results for the $1$-step binary markets (equation \eqref{r1s}), we obtain that $\lambda>\lambda_*$.
\end{proof}
\begin{remark}
If we assume that $\la_c=0$, then $\al_n(x)\geq1\geq\beta_n(x)$ for all $n\in\{1,...,N\}$ and $x\in E_{n-1}$, and, by definition, the arbitrage
opportunities disappear when we introduce arbitrarily small transaction costs. If in addition $\al_n(x)>1>\beta_n(x)$ for all $n\in\{1,...,N\}$ and
$x\in E_{n-1}$, then there are no arbitrage opportunities in the frictionless market.\\
\end{remark}

\section{\texorpdfstring{Necessary and sufficient conditions on the measures inducing $\lambda$-CPS}{}}\label{s2}

In this section we study necessary and sufficient conditions for a probability measure to be in $\Ms(\lambda)$. This is the starting point to understand
the nature of the $\lambda$-CPS for binary markets. We will see how the martingale property imposes constraints in the bid-ask spread intervals, and how to
deduce from these constraints a necessary condition to belong to $\Ms(\lambda)$, which turns out to be also sufficient.
\subsection{\texorpdfstring{Effective bid-ask spread of $S$}{}}
The goal of this paragraph is to show that if $(\Qs, \tilde{S})$ is a $\lambda$-CPS for the process $S$, then $\tilde{S}$ verifies a condition which is, in general,
stronger than \eqref{CPSd}.\\

\noindent To make this idea clear, we introduce the next lemma, which shows that, by using the properties of the
$\lambda$--CPS, property \eqref{CPSd} at time $n$ implies a more restrictive condition at time $n-1$.
\begin{lemma}\label{ab}
 Let $\lambda\in[0,1]$ and $(\cQ,\tilde{S})\in\cS^{\lambda}$. If, for $n>1$ and $y\in E_{n-1}$, there exists
 $a,b,\tilde{a},\tilde{b}$ strictly positive such that
\begin{equation}\label{yu}
  \frac{\tilde{S}_n(y\star u)}{S_n(y\star u)}\in\left[(1-\lambda)a,b\right]\quad\textrm{and}\quad \frac{\tilde{S}_n(y\star d)}{S_n(y\star d)}\in\left[(1-\lambda)\tilde{a},\tilde{b}\right],
\end{equation}
then
$$
\frac{\tilde{S}_{n-1}(y)}{S_{n-1}(y)}\in\left[(1-\lambda)(\overline{a}\vee1), \overline{b}\wedge1\right]
$$
where
$$\overline{a}=q_n(y)\alpha_n(y)a+(1-q_n(y))\beta_n(y)\tilde{a}$$
and
$$\overline{b}=q_n(y)\alpha_n(y)b+(1-q_n(y))\beta_n(y)\tilde{b}.$$
\end{lemma}
\begin{proof}
Let $n>1$ and $y\in E_{n-1}$ such that \eqref{yu} hold true. It is enough to prove that
$\frac{\tilde{S}_{n-1}(y)}{S_{n-1}(y)}\in\left[(1-\lambda)\overline{a}, \overline{b}\right]$. Indeed, if this
would be the case, then the desired result follows from the fact that
$(1-\lambda)S_{n-1}(y)\leq \tilde{S}_{n-1}(y)\leq S_{n-1}(y)$.\\
By the martingale property, we obtain that:
\begin{equation}\label{eqmk}
 \tilde{S}_{n}(y\star d)=\frac{\tilde{S}_{n-1}(y)- q_{n}(y)\,\tilde{S}_{n}(y\star u)}{1-q_{n}(y)},
\end{equation}
\noindent which, combined with \eqref{yu}, gives us that:
\begin{equation}\label{eqpk1}
\frac{\tilde{S}_{n-1}(y)-(1-q_{n}(y))\tilde{b}S_{n}(y\star d)}{q_{n}(y)} \leq\tilde{S}_{n}(y\star u)
\end{equation}
\noindent and
\begin{equation}\label{eqpk2}
\tilde{S}_{n}(y\star u)\leq \frac{\tilde{S}_{n-1}(y)- (1-\lambda)(1-q_{n}(y))\tilde{a}S_{n}(y\star d)}{q_{n}(y)}
\end{equation}
Then, \eqref{eqpk1} together with \eqref{yu} implies that the left hand side of \eqref{eqpk1} is
smaller or equal than $bS_{n}(y\star u)$. As well, \eqref{eqpk2} combined with
\eqref{yu} implies that the right hand side of \eqref{eqpk2} is bigger or equal than
$(1- \lambda)aS_{n}(y\star u)$.
From this and by using that $S_{n}(y\star u)=\alpha_{n}(y)S_{n-1}(y)$ and that
$S_{n}(y\star d)=\beta_{n}(y)S_{n-1}(y)$, we obtain that:
\begin{equation*}
(1-\lambda)\, \overline{a}\leq \frac{\tilde{S}_{n-1}(y)}{S_{n-1}(y)}\leq \overline{b}.
\end{equation*}
\end{proof}

Now, starting from the result presented in the above lemma, but iterated for every time point,
we introduce, for each $n\in\{1,...,N+1\}$, the functions $\rho_n^+$ and $\rho_n^-$ as follows.

The functions $\rho_{N+1}^+,\rho_{N+1}^-: E_{N}\rightarrow \Rb_+$ are defined by putting:
$$\rho_{N+1}^+=\rho_{N+1}^-\equiv 1.$$
For $n\in\{1,...,N\}$, the functions $\rho_{n}^+,\rho_{n}^-:\Ps_1(\Omega)\times E_{n-1}\rightarrow \Rb_+$ are defined by means of a backward recurrence relation.
More precisely, for each $\Qs\in\Ps_1(\Omega)$ and $x\in E_{n-1}$, we put:
$$\rho_{n}^+(\Qs,x)=1\wedge\left[\,q_{n}(x)\,\alpha_{n}(x)\,\rho_{n+1}^+(\Qs,x\star u)+(1-q_{n}(x))\,\beta_{n}(x)\,\rho_{n+1}^+(\Qs,x\star d)\right],$$
and
$$\rho_{n}^-(\Qs,x)=1\vee\left[\,q_{n}(x)\,\alpha_{n}(x)\,\rho_{n+1}^-(\Qs,x\star u)+(1-q_{n}(x))\,\beta_{n}(x)\,\rho_{n+1}^-(\Qs,x\star d)\right].$$
For $n=N$ we need to replace $(\Qs,x\star\cdot)$ by $x\star\cdot$ in this recurrence relation.\\

In the following proposition, we establish that, as expected according to its construction, the quantities
$(1-\lambda)\rho_{n+1}^-(\Qs,y)S_n(y)$ and $\rho_{n+1}^+(\Qs,y)S_n(y)$
represent the extremities of the effective bid-ask spread interval at the time $n$ at the position $y$.

\begin{proposition}[Effective bid-ask spread of $S$]\label{pcn}
  If $\lambda\in[0,1]$ and $(\Qs,\tilde{S})\in\Ss^\lambda$, then, for each $n\in\{0,...,N\}$ and $y\in E_{n}$:
$$\frac{\tilde{S}_{n}(y)}{S_{n}(y)}\in\left[(1-\lambda)\rho_{n+1}^-(\Qs,y),\rho_{n+1}^+(\Qs,y)\right].$$
\end{proposition}
\begin{proof}
We prove the result by backward recurrence. For $n=N$, the statement is true by the definitions of $\lambda$-CPS, $\rho_{N+1}^{+}$ and $\rho_{N+1}^{-}$.

Now, we suppose that the result is true for some $n\in\{1,...,N\}$ and we prove it for $n-1$. We fix $y\in E_{n-1}$.
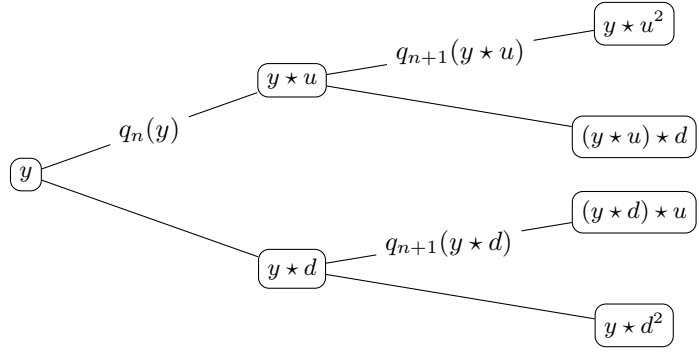
\begin{center}
\begin{tikzpicture}[grow=right]
\tikzstyle{level 0}=[rectangle,rounded corners, draw,level distance=20mm]
\tikzstyle{level 1}=[rectangle,rounded corners, draw,level distance=35mm, sibling distance=25mm]
\tikzstyle{level 2}=[rectangle,rounded corners, draw,level distance=45mm, sibling distance=15mm]
\node[level 0] {\small{$y$}}
child{node[level 1]{\small{$y\star d$}}
  child{node[level 2]{\small{$y\star d^2$}}}
  child{ node[level 2] {\small{$(y\star d)\star u$}}
         edge from parent node[fill=white] {$q_{n+1}(y\star d)$}
       }
}
child{node[level 1]{\small{$y\star u$}}
  child {node[level 2]{\small{$(y\star u)\star d$}}}
  child {node[level 2] {\small{$y\star u^2$}}
         edge from parent node[fill=white] {$q_{n+1}(y\star u)$}
        }
edge from parent node[fill=white] {$q_{n}(y)$}
};
\end{tikzpicture}

\end{center}
We know by the recurrence hypothesis that:
\begin{equation}\label{eqbak1}
\frac{\tilde{S}_{n}(y\star u)}{S_{n}(y\star u)}\in\left[(1-\lambda)\rho_{n+1}^-(\Qs,y\star u),\rho_{n+1}^+(\Qs,y\star u)\right].
\end{equation}
and
\begin{equation}\label{eqbak2}
\frac{\tilde{S}_{n}(y\star d)}{S_{n}(y\star d)}\in\left[(1-\lambda)\rho_{n+1}^-(\Qs,y\star d),\rho_{n+1}^+(\Qs,y\star d)\right].
\end{equation}
The result follows immediately by applying Lemma~\ref{ab}.
\end{proof}

\subsection{\texorpdfstring{Some properties of the functions $\rho_n^+$ and $\rho_n^-$}{}}
As we have seen in Proposition \ref{pcn} the functions $\rho^+$ and $\rho^-$ encode the effect of the martingale property in the
dynamic of the bid-ask spread intervals. Therefore, it seems important to understand their nature. To this end we present in this paragraph
some useful properties of these functions.

\begin{lemma}\label{la}Let $\Qs$ be a probability measure equivalent to $P$. For each $n\in \{1,...,N\}$ and $x\in E_{n-1}$.
\begin{enumerate}
\item If $\rho_n^+(\Qs,x)=1$, then $\alpha_n(x)>1$.
\item If $\rho_n^-(\Qs,x)=1$, then $\beta_n(x)<1$.
\item If $\rho_n^+(\Qs,x)=\rho_n^-(\Qs,x)=1$, then for $y\in\{x\star u,x \star d\}$:
$$\rho_{n+1}^+(\Qs,y)=\rho_{n+1}^-(\Qs,y)=1\quad\textrm{and}\quad q_n(x)=\frac{1-\beta_n(x)}{\alpha_n(x)-\beta_n(x)}.$$
\end{enumerate}
\end{lemma}
\begin{proof}
For $n=N$, the statements follow immediately as $\alpha_N(x)-\beta_N(x)>0$ and $\rho_{N+1}^+(\Qs,x)=\rho_{N+1}^-(\Qs,x)=1$. Now let $n\in \{1,...,N-1\}$ and $x\in E_{n-1}$.\\
(1) As $\rho_n^+(\Qs,x)=1$, we have that
\begin{equation}\label{rho+}
q_{n}(x)\,\alpha_{n}(x)\,\rho_{n+1}^+(\Qs,x\star u)+(1-q_{n}(x))\,\beta_{n}(x)\,\rho_{n+1}^+(\Qs,x\star d)\geq1.
\end{equation}
From the above inequality we immediately deduce that $\alpha_n(x)>1$. Indeed, if we assume that $\alpha_n(x)\leq 1$ and since $\rho_{n+1}^+(\Qs,y)\leq1$, with $y\in\{x\star u,x\star d\}$,
and $\beta_n(x)<\alpha_n(x)$, then (\ref{rho+}) implies that:
$$1\leq q_{n}(x)\,\alpha_{n}(x)\,\rho_{n+1}^+(\Qs,x\star u)+(1-q_{n}(x))\,\beta_{n}(x)\,\rho_{n+1}^+(\Qs,x\star d)<1.$$
Hence, we obtained a contradiction.\\
(2) When $\rho_n^-(\Qs,x)=1$, it follows that:
\begin{equation}\label{rho-}
q_{n}(x)\,\alpha_{n}(x)\,\rho_{n+1}^-(\Qs,x\star u)+(1-q_{n}(x))\,\beta_{n}(x)\,\rho_{n+1}^-(\Qs,x\star d)\leq1.
\end{equation}
Like previously, we can directly deduce, using the fact that $\rho_{n+1}^-(\Qs,y)\geq1$ with $y\in\{x\star u,x\star d\}$, that $\beta_n(x)<1$.\\
(3) If $\rho_n^+(\Qs,x)=\rho_n^-(\Qs,x)=1$, the inequalities (\ref{rho+}) and (\ref{rho-}) hold simultaneously. Using that
$\rho_{n+1}^+(\Qs,y)\leq1$ and $\rho_{n+1}^-(\Qs,y)\geq1$ for $y\in\{x\star u,x\star d\}$, we can see that this is only possible if
$\rho_{n+1}^+(\Qs,y)=\rho_{n+1}^-(\Qs,y)=1$ for $y\in\{x\star u,x \star d\}$.
We only have to plug this in (\ref{rho+}) and (\ref{rho-}) to obtain that:
$$q_n(x)\alpha_n(x)+(1-q_n(x))\beta_n(x)=1$$
which clearly implies that $q_n(x)=\frac{1-\beta_n(x)}{\alpha_n(x)-\beta_n(x)}$.
\end{proof}

In order to lighten some of the proofs, we introduce some extra notations. For each $n\in\{1,...,N\}$ and $x\in E_{n-1}$, we set:
$$r_n^+(\Qs,x)=q_{n} (x)\alpha_{n} (x)\rho_{n+1}^+(\Qs,x\star u)+ (1-q_{n} (x))\beta_{n}(x)\rho_{n+1}^+(\Qs,x\star d)$$
and
$$r_n^-(\Qs,x)=q_{n} (x)\alpha_{n} (x)\rho_{n+1}^-(\Qs,x\star u)+ (1-q_{n} (x))\beta_{n}(x)\rho_{n+1}^-(\Qs,x\star d).$$
Using these notations, we remark that:
\begin{equation}\label{rvsrho}
\rho_n^+(\Qs,x)=1\wedge r_n^+(\Qs,x)\textrm{ and }\rho_n^-(\Qs,x)=1\vee r_n^-(\Qs,x).
\end{equation}
Note that, from these identities and the definitions, we can deduce the following chain of inequalities:
\begin{equation}\label{irhpmr}
\rho_n^+(\Qs,x)\leq r_n^+(\Qs,x)\leq r_n^-(\Qs,x)\leq \rho_n^-(\Qs,x).
\end{equation}
\begin{remark}[Continuity]
For each \mbox{$n\in\{1,...,N+1\}$} and $x\in E_{n-1}$, the functions $\rho_n^+(\cdot,x)$ and $\rho_n^-(\cdot,x)$ are continuous. The proof follows easily by backward recurrence. Indeed,
for $n=N+1$, we have that $\rho_{N+1}^+=\rho_{N+1}^-\equiv 1$ and hence they are continuous. For the induction step, by assuming that $\rho_{n+1}^+$
and $\rho_{n+1}^-$ are continuous, it follows directly from the definition that $\rho_n^+$ and $\rho_n^-$ are as well continuous. As a consequence the functions $r_n^+(\cdot,x)$ and $r_n^-(\cdot,x)$ for each
$n\in\{1,...,N\}$ and $x\in E_{n-1}$ are also continuous.
\end{remark}
\begin{remark}\label{r1}
Note that for fixed $n\in\{1,...,N\}$ and $x\in E_{n-1}$, the quantities $\rho_{n}^+(\Qs,x)$ and $\rho_{n}^-(\Qs,x)$ depend only on the coordinates of $\Qs$ associated to the nodes of the sub-tree generated by $x$ and not on the whole probability $\Qs$.
\end{remark}

\subsection{Necessary and sufficient condition}
In this paragraph we establish necessary and  sufficient conditions for a measure to induce a $\lambda$-CPS.

In order to provide a necessary condition, we define for $n\in\{1,...,N\}$ and $x\in E_{n-1}$:
$$\Delta_n^\lambda(\Qs,x)\equiv \rho_n^+(\Qs,x)-(1-\lambda)\rho_n^-(\Qs,x).$$
Note that $\Delta_n^\lambda(\Qs,x)S_{n-1}(x)$ is the length of the effective bid-ask spread interval at the time $n-1$ at the position $x$.
Thus, the following necessary condition appears in a natural way.
\begin{corollary}[Necessary condition]\label{ccn}
If $\lambda\in[0,1]$ and $\Qs\in\Ms(\lambda)$, then for all $n\in\{1,...,N\}$ and $x\in E_{n-1}$:
$$\Delta_n^\lambda(\Qs,x) \geq 0.$$
\end{corollary}
\begin{proof}
Direct from Proposition \ref{pcn}
\end{proof}
Now, we establish a sufficient condition, which is in fact the converse of Corollary \ref{ccn}.
\begin{proposition}[Sufficient condition]\label{cpcn}
If for $\lambda>0$ there exists $\Qs\sim P$ such that for all $n\in\{1,...,N\}$ and $x\in E_{n-1}$:
$$\Delta_n^\lambda(\Qs,x) \geq 0,$$
then $\Qs\in\Ms(\lambda)$.
\end{proposition}
\begin{proof}
We fix $\lambda>0$ and $\Qs\sim P$ such that for all $n\in\{1,...,N\}$ and $x\in E_{n-1}$:
$$\Delta_n^\lambda(\Qs,x) \geq 0,$$
and we will construct inductively a process $\tilde{S}={(\tilde{S}_n)}_{n=0}^N$ such that $(\Qs,\tilde{S})$ is a $\lambda$-CPS.

We start by taking:
\begin{equation}\label{ecpnc1}
\tilde{S}_0(x_0)=\tilde{s}_0\in[(1-\lambda)\rho_1^-(\Qs,x_0)s_0,\rho_1^+(\Qs,x_0)s_0].
\end{equation}
We set:
$$d_1(x_0)=\rho_1^+(\Qs,x_0)-\frac{\tilde{s}_0}{s_0},$$
and we note that $0\leq d_1(x_0)\leq \Delta_1^\lambda(\Qs,x_0)$.

Now, for $n\in \{1,...,N\}$, we suppose that we have constructed a $(n-1)$-step martingale $\tilde{S}={(\tilde{S}_k)}_{k=0}^{n-1}$ verifying:
\begin{equation}\label{basc}
 \frac{\tilde{S}_{k}(z)}{S_{k}(z)}\in[(1-\lambda)\rho_{k+1}^-(\Qs,z),\rho_{k+1}^+(\Qs,z)],
\end{equation}
for all $k\in\{0,...,n-1\}$ and $z\in E_{k}$. We note that, by defining:
\begin{equation}\label{eddk}
 d_{k+1}(z)=\rho_{k+1}^+(\Qs,z)-\frac{\tilde{S}_{k}(z)}{S_{k}(z)},\qquad k\in\{0,...,n-1\},\,z\in E_{k},
\end{equation}
condition \eqref{basc} is equivalent to:
\begin{equation}\label{eqdr}
0\leq d_{k+1}(z)\leq \Delta_{k+1}^\lambda(\Qs,z).
\end{equation}

The goal is to extend $\tilde{S}$ to a $n$-step martingale satisfying \eqref{basc} for $k=n$.
With this purpose in mind, we fix $y\in E_{n-1}$ and we aim to construct $\tilde{S}_{n}(y\star u)$ and $\tilde{S}_{n}(y\star d)$. Since the extension
of $\tilde{S}$ must verify the $\Qs$-martingale property, we need only to choose in a proper way $\tilde{S}_{n}(y\star u)$ and then to put:
\begin{equation}\label{emart}
 \tilde{S}_n(y\star d)=\frac{\tilde{S}_{n-1}(y)-q_n(y)\tilde{S}_n(y\star u)}{1-q_n(y)}.
\end{equation}
So, we need to prove that we can choose $\tilde{S}_{n}(y\star u)$ in the associated effective bid-ask spread interval, in such a way that $\tilde{S}_{n}(y\star d)$
defined by means of \eqref{emart} is also in the corresponding effective bid-ask spread interval. Equivalently, we need to show that we can choose $d_{n+1}(y\star u)$ such that:
\begin{equation}\label{ed1}
0\leq d_{n+1}(y\star u)\leq \Delta_{n+1}^\lambda(\Qs,y\star u),
\end{equation}
and, by setting:
\begin{equation}\label{esnyu}
\tilde{S}_n(y\star u)=\left(\rho_{n+1}^+(\Qs,y\star u)-d_{n+1}(y\star u)\right)S_n(y\star u),
\end{equation}
we have that $\tilde{S}_{n}(y\star d)$ defined by \eqref{emart} verifies:
\begin{equation}\label{esnyd}
\frac{\tilde{S}_{n}(y\star d)}{S_{n}(y\star d)}\in[(1-\lambda)\rho_{n+1}^-(\Qs,y\star d),\rho_{n+1}^+(\Qs,y\star d)].
\end{equation}
To this end, we will express condition \eqref{esnyd} in terms of $d_{n+1}(y\star u)$ and then prove that this condition is compatible with \eqref{ed1}.

Plugging \eqref{esnyu} in \eqref{emart} and using \eqref{eddk} for $k=n-1$, we obtain that:
\begin{equation*}
 \frac{\tilde{S}_{n}(y\star d)}{S_{n}(y\star d)}=\frac{q_n(y)\alpha_n(y)}{(1-q_n(y))\beta_n(y)}\left[\frac{\rho_{n}^+(\Qs,y)-d_{n}(y)}{q_n(y)\alpha_n(y)}-\rho_{n+1}^+(\Qs,y\star u)+d_{n+1}(y\star u)\right],
\end{equation*}
and then, condition \eqref{esnyd} becomes:
\begin{equation}\label{cdn12}
r_n(\Qs,y)- \frac{(1-q_n(y))\beta_n(y) \Delta_{n+1}^\lambda(\Qs,y\star d)}{q_n(y)\alpha_n(y)}\leq d_{n+1}(y\star u)\leq r_n(\Qs,y),
\end{equation}
where
$$r_n(\Qs,y)=\frac{r_n^+(\Qs,y)-\rho_n^+(\Qs,y)+d_n(y)}{q_n(y)\alpha_n(y)}.$$
Note that condition \eqref{cdn12} makes sense, because $\Delta_{n+1}^\lambda(\Qs,y\star d)\geq 0$ by hypothesis. Similarly, condition \eqref{ed1} make sense since $\Delta_{n+1}^\lambda(\Qs,y\star u)\geq 0$.
Moreover, as $r_n^+(\Qs,y)\geq\rho_n^+(\Qs,y)$ and $d_n(y)\geq 0$, we have that $r_n(\Qs,y)\geq0$. It follows that the right hand side of the inequality \eqref{cdn12}
is compatible with the left hand side of inequality \eqref{ed1}. It remains to prove that the right hand side of inequality \eqref{ed1} is compatible with the left hand side of \eqref{cdn12}, that means:
$$r_n(\Qs,y)- \frac{(1-q_n(y))\beta_n(y) \Delta_{n+1}^\lambda(\Qs,y\star d)}{q_n(y)\alpha_n(y)}\leq\Delta_{n+1}^\lambda(\Qs,y\star u),$$
which is equivalent to:
$$q_n(y)\alpha_n(y)r_n(\Qs,y)\leq r_n^+(\Qs,y)-(1-\lambda)r_n^-(\Qs,y),$$
which is also equivalent to:
$$d_n(y)\leq\Delta_{n}^\lambda(\Qs,y)+(1-\lambda)\left(\rho_n^-(\Qs,y)-r_n^-(\Qs,y) \right),$$
which is true by \eqref{eqdr} and the fact that $\rho_n^-(\Qs,y)\geq r_n^-(\Qs,y)$. We conclude the existence of $d_{n+1}(y\star u)$ verifying \eqref{ed1} and
\eqref{cdn12} and then, by means of \eqref{esnyu} and \eqref{emart}, the existence of $\tilde{S}_{n}(y\star u)$ and $\tilde{S}_{n}(y\star d)$ verifying the desired properties.
Repeating the procedure for each $y\in E_{n-1}$, we succeed to extend $\tilde{S}$ to a $n$-step $\lambda$-CPS.\\
Thus, thanks to a forward recurrence, we can construct $\tilde{S}$ such that $(\Qs,(\tilde{S}_n)_{0\leq n\leq N})$ is a $\lambda$-CPS. The result was proved.
\end{proof}

\section{Characterizations}\label{s3}
In the previous section, we have found a necessary and sufficient condition for a measure $\Qs$ to induce a $\lambda$-CPS. Based on
this condition, in this section, we obtain a characterization for the smallest transaction cost $\lambda_c$ necessary to remove arbitrage opportunities.
Similarly, we obtain a characterization of the set $\Ms(\lambda)$ as the preimage of an interval of a continuous function on the space of probability
measures equivalent to $P$. We end this section studying in depth the set $\Ms(\lambda_c)$.

Before to start with the mentioned characterizations, we define the function $\rho:\Ps_1(\Omega)\rightarrow (0,1]$ by putting:
$$\rho(\Qs)=\min\limits_{n\in\{1,...,N\}}\left[\min\limits_{x\in E_{n-1}}\frac{\rho_n^+(\Qs,x)}{\rho_n^-(\Qs,x)}\right],\quad \Qs\in\Ps_1(\Omega),$$
which will play a crucial role in what follows. Note that this function is continuous, because it is the minimum of a finite number of continuous functions.

\subsection{\texorpdfstring{Characterization of $\lambda_c$}{}}
\begin{theorem}\label{thmc}
We have that:
$$\lambda_c=1-\sup\limits_{Q\sim P}\rho(Q).$$
\end{theorem}
\begin{proof}
We start proving that $\lambda_c\geq 1-\sup\limits_{Q\sim P}\rho(Q)$. By definition of $\lambda_c$, for each $\lambda>\lambda_c$ there exists a $\lambda$-CPS: $(\Qs,\tilde{S})$.
By using Proposition \ref{pcn}, we deduce that for all $n\in\{1,...,N\}$ and $x\in E_{n-1}$:
$$(1-\lambda)\rho_n^-(\Qs,x)\leq \rho_n^+(\Qs,x).$$
We divide by $\rho_n^-(\Qs,x)$ both sides of this inequality and we take the minimum on all $n\in\{1,...,N\}$ and $x\in E_{n-1}$ to obtain:
$$1-\lambda\leq \rho(\Qs)\leq \sup\limits_{Q\sim P}\rho(Q),$$
and then $\lambda\geq 1- \sup\limits_{Q\sim P}\rho(Q)$. The statement follows because the last inequality is true for all $\lambda>\lambda_c$.

Now, we prove that $\lambda_c\leq 1-\sup\limits_{Q\sim P}\rho(Q)$. For this, we take $\lambda<\lambda_c$. By Proposition \ref{cpcn}, for each probability
$\Qs\sim P$, there exists $n\in\{1,...,N\}$ and $x\in E_{n-1}$ such that:
$$1-\lambda>\frac{\rho_n^+(\Qs,x)}{\rho_n^-(\Qs,x)}\geq \rho(\Qs),$$
and then $1-\lambda>\rho(\Qs)$. Since this inequality is true for all $\Qs\sim P$, we deduce that:
$$\lambda<1-\sup\limits_{Q\sim P}\rho(Q).$$
The last inequality being true for all $\lambda<\lambda_c$, the result follows.
\end{proof}
We end this paragraph with the following representation of $\lambda_c$, which is slightly different to that obtained in Theorem \ref{thmc}.
\begin{corollary}\label{lac}
 We have that:
$$\lambda_c=1-\sup\limits_{Q\in\Ps_1(\Omega)}\rho(Q).$$
\end{corollary}
\begin{proof}
By Theorem \ref{thmc}, we know that $\lambda_c=1-\sup\limits_{Q\sim P}\rho(Q)$. Since:
$$\{Q: Q\sim P\}\subseteq\Ps_1(\Omega),$$
we have that
$$\sup\limits_{Q\sim P}\rho(Q)\leq \sup\limits_{Q\in\Ps_1(\Omega)}\rho(Q)=\rho(\Qs^*).$$
If $\Qs^*\sim P$, the result follows. If this is not the case, we define, for each $0<\varepsilon<1$, a probability measure $\Qs^\varepsilon\sim P$,
by setting for all $n\in \{1,...,N\}$ and $y\in E_{n-1}$:
\begin{equation*}
 q_n^{\varepsilon}(y)=\begin{cases}
q_n^*(y) & \text{if } q_n^*(y)\in(0,1),\\
\varepsilon & \text{if } q_n^*(y) = 0,\\
1-\varepsilon & \text{if } q_n^*(y) =1,
\end{cases}
\end{equation*}
where $q_n^*(y):=q_n(\Qs^*,y)$ and $q_n^{\varepsilon}(y):=q_n(\Qs^{\varepsilon},y)$ with $q_n(\cdot,y)$ given by \eqref{qn}.
As $\rho_n^+$ and $\rho_n^-$ are continuous functions for each $n\in\{1,\ldots,N+1\}$, it follows that $\rho$ is as well continuous, and
therefore
$$\rho(\Qs^{\varepsilon})\xrightarrow[\varepsilon\to0]{} \rho(\Qs^*).$$
Now, since $\Qs^{\varepsilon}\sim P$,
$$\rho(\Qs^{\varepsilon})\leq \sup\limits_{Q\sim P}\rho(Q)\leq \rho(\Qs^*),$$
and we obtain that indeed $\sup\limits_{Q\sim P}\rho(Q)=\rho(\Qs^*)$.
\end{proof}
\begin{remark}
The advantage of Corollary \ref{lac} with respect to Theorem \ref{thmc} lies in the fact that the supremum of $\rho$ in the set $\{\Qs: \Qs\sim P\}$ is
not necessarily reached while the supremum of $\rho$ taken in $\Ps_1(\Omega)$ it is. This will be particularly useful in order to obtain good upper bounds
for $\lambda_c$.
\end{remark}

\begin{remark}\label{r6}
 Note that, if $\al_n(x)\geq1\geq\beta_n(x)$ for all $n\in\{1,...,N\}$ and $x\in E_{n-1}$, then $\la_c=0$. Indeed, if we define the probability $\Qs^*$ by:
 $$q_n(\Qs^*,x)=\frac{1-\beta_n(x)}{\alpha_n(x)-\beta_n(x)},\qquad n\in\{1,...,N\},\, x\in E_{n-1},$$
 we can prove easily that:
 $$\rho(\Qs^*)=1.$$
Using Corollary \ref{lac} we deduce that $\lambda_c=0$.
 \end{remark}
\subsection{\texorpdfstring{Characterization of $\mathcal{M}(\lambda)$}{}}
\begin{theorem}\label{thmmlc}
 For each $\lambda\in[0,1]$, we have that:
 $$\mathcal{M}(\lambda)=\left\{ \Qs\sim P:\,\rho(\Qs)\geq 1-\lambda\right\}=\rho_*^{-1}\left([1-\lambda,1]\right),$$
 where $\rho_*$ denotes the restriction of $\rho$ to the set of all probability measures $\Qs\sim P$.
\end{theorem}
\begin{proof}
We fix $\lambda\in[0,1]$. We prove first that:
 $$\mathcal{M}(\lambda)\subseteq\left\{ \Qs\sim P:\,\rho(\Qs)\geq 1-\lambda\right\}$$
Indeed, if $\Qs\in\mathcal{M}(\lambda)$, then by Corollary \ref{ccn}, we conclude that for all $n\in\{1,...,N\}$ and
$x\in E_{n-1}$: $\Delta_n^\lambda(\Qs,x)\geq 0$, hence by definition:
$$1-\lambda\leq\frac{\rho_n^+(\Qs,x)}{\rho_n^-(\Qs,y)}\leq 1.$$
It follows that $\rho(\Qs)\geq 1-\lambda$, and this proved the first inclusion.

It remains to prove that:
 $$\mathcal{M}(\lambda)\supseteq\left\{ \Qs\sim P:\,\rho(\Qs)\geq 1-\lambda\right\}.$$
In order to do this, we take $\Qs\sim P$ such that $\rho(\Qs)\geq 1-\lambda$. This implies that for all $n\in\{1,...,N\}$ and
$x\in E_{n-1}$:
$$\frac{\rho_n^+(\Qs,x)}{\rho_n^-(\Qs,y)}\geq 1-\lambda,$$
and then $\Delta_n^\lambda(\Qs,x)\geq 0$. Using Proposition \ref{cpcn}
we conclude that $\Qs\in\mathcal{M}(\lambda)$. The proof is finished.
\end{proof}
\subsection{\texorpdfstring{Characterization of $\mathcal{M}(\lambda_c)$}{}}
The Theorem \ref{thmmlc} provides for each $\lambda\in[0,1]$ an expression for the set $\Ms(\lambda)$. Obviously, when $\lambda<\lambda_c$
we can be more precise and say that $\Ms(\lambda)=\emptyset$. When $\lambda>\lambda_c$, we can say that $\Ms(\lambda)\neq\emptyset$. But in the transition phase, i.e., when
$\lambda=\lambda_c$, we can not say a priori if $\Ms(\lambda)$ is empty or not. That is the goal of this paragraph.

We start with the next lemma which is a stronger version of the Theorem \ref{thmmlc} for the special case $\lambda=\lambda_c$.
\begin{lemma}\label{qlc}
$\Qs\in\Ms(\lambda_c)$ if and only if $\Qs\sim P$ and $\lambda_c=1-\rho(\Qs)$.
\end{lemma}

\begin{proof}
If $\Qs\in\Ms(\lambda_c)$, then by using Theorems \ref{thmc} and \ref{thmmlc}, we obtain that:
$$\rho(\Qs)\geq 1-\lambda_c=\sup\limits_{Q\sim P}\rho(Q),$$
which implies that $\rho(\Qs)=1-\lambda_c$.\\
The other implication follows from Theorem \ref{thmmlc}.
\end{proof}

We know from this lemma, that if  $\Qs\in\Ms(\lambda_c)$, then there exists $n\in\{1,...,N\}$ and $x\in E_{n-1}$ such that:
\begin{equation}\label{minp}
 \rho_n^+(\Qs,x)=(1-\lambda_c)\rho_n^{-}(\Qs,x).
\end{equation}
The nodes verifying this identity will be particularly interesting in order to characterize $\Ms(\lambda_c)$. For this reason, we
define for each $\Qs\in\Ps(\Omega)$, the sets:
$$A_n(\Qs)=\left\{x\in E_{n-1}:\, \rho_n^+(\Qs,x)=(1-\lambda_c)\rho_n^{-}(\Qs,x)\right\},\qquad n\in\{1,...,N\},$$
and we put:
$$\nu(\Qs)=\sum\limits_{n=1}^N|A_n(\Qs)|.$$
By definition, $\nu(\Qs)$ is the number of points verifying \eqref{minp}. It follows from the previous lemma that if $\Qs\in\Ms(\lambda_c)$ then $\nu(\Qs)>0$. In that case, we can define:
$$k_\Qs=\max\{n\in\{1,...,N\}:\, A_n(\Qs)\neq\emptyset\}.$$
\begin{lemma}\label{laux}
 If $\Qs\in\Ms(\lambda_c)$, then for all $x\in A_{k_\Qs}(\Qs)$:
 $$r_{k_\Qs}^+(\Qs,x)>1 \,\textrm{  or  }\,r_{k_\Qs}^-(\Qs,x)<1.$$
\end{lemma}
\begin{proof}
To simplify the notations, we set $k=k_\Qs$. Now, we fix $x\in A_k(\Qs)$. If $r_k^+(\Qs,x)\leq 1$, then, by maximality of $k$, we deduce that:
\begin{align*}
r_k^+(\Qs,x)&=q_{k} (x)\alpha_{k} (x)\rho_{k+1}^+(\Qs,x\star u)+ (1-q_{k} (x))\beta_{k}(x)\rho_{k+1}^+(\Qs,x\star d)\\
 &>(1-\lambda_c)\left(q_{k} (x)\alpha_{k} (x)\rho_{k+1}^-(\Qs,x\star u)+ (1-q_{k} (x))\beta_{k}(x)\rho_{k+1}^-(\Qs,x\star d)\right)\\
&=(1-\lambda_c) r_k^-(\Qs,x).
 \end{align*}
On the other hand, since $x\in A_k(\Qs)$ and $r_k^+(\Qs,x)\leq 1$:
$$r_k^+(\Qs,x)=\rho_k^+(\Qs,x)= (1-\lambda_c)[1\vee r_k^-(\Qs,x)].$$
Combining this identity with the previous inequality, we obtain the result.\\
\end{proof}
Note that if $\Qs\in\Ms(\lambda_c)$, from the lemma above, we can deduce that for each  $x\in A_{k_\Qs}(\Qs)$, we have either $\rho_{k_\Qs}^+(\Qs,x)=1$ or $\rho_{k_\Qs}^-(\Qs,x)=1$. However, the last assertion is not a priori stable
under small perturbations on $\Qs$ while the assertion in the lemma it is. More precisely, if we start with a point satisfying $r_k^+(\Qs,x)>1$ (respectively $r_k^{-}(\Qs,x)<1$), then by continuity, we can find $\varepsilon>0$ such that if
$d_\infty(\widehat{\Qs},\Qs)\leq \varepsilon$, then $r_k^+(\widehat{\Qs},x)>1$ (respectively $r_k^-(\widehat{\Qs},x)<1$).

Now, we fix $k=k_\Qs$ and $x\in A_{k}(\Qs)$. We will be interested on the behavior under small perturbations on $\Qs$ of the ratio:
\begin{equation}\label{ratio}
 \frac{\rho_{k}^+(\Qs,x)}{\rho_{k}^-(\Qs,x)}=1-\lambda_c.
\end{equation}
We know from the previous discussion that either the numerator or the denominator remains equal to one under small perturbations. If in addition
$\lambda_c>0$, we conclude that either $\rho_{k}^+(\Qs,x)<1$ or $\rho_{k}^-(\Qs,x)>1$. Using this, the next lemma proves in particular that
we can do small perturbations in such a way that the ratio in \eqref{ratio} increase.

\begin{lemma}\label{lmon}Let $\Qs\sim P$, $k\in\{1,...,N\}$ and $y\in E_{k-1}$. If $\rho_k^+(\Qs,y)<1$ (respectively $\rho_k^-(\Qs,y)>1$), then there exist $\ell\geq k$ and $(z,y)\in E_{\ell-1}$ such that for every $\varepsilon>0$ there exists $\Qs^\varepsilon\sim P$ verifying:
\begin{enumerate}
\item\label{lc1} $\lvert q_\ell^\varepsilon(z,y)-q_\ell(z,y)\rvert\leq \varepsilon.$
\item\label{lc2} $q_n^\varepsilon(x)=q_n(x)\textrm{ if and only if }  n\neq \ell \textrm{ or } x\neq(z,y).$
\item\label{lc3} $\rho_k^+(\Qs^\varepsilon,y)>\rho_k^+(\Qs,y)\,\, (\textrm{resp. } \rho_k^-(\Qs^\varepsilon,y)<\rho_k^-(\Qs,y)).$
\end{enumerate}
We used the notation $q_n^\varepsilon(\cdot):=q_n(\Qs^\varepsilon,\cdot)$, where $q_n(\Qs^\varepsilon,\cdot)$ is given in \eqref{qn}.
\end{lemma}
\begin{proof}
We give the proof for the assertion concerning the case $\rho_k^+(\Qs,y)<1$ (the proof for the case $\rho_k^-(\Qs,y)>1$ is analogous). We prove this by means of a backward induction on the level $k$.
So, we start with the proof for $k=N$.\\
If  $\rho_N^+(\Qs,y)<1$, then:
$$\rho_N^+(\Qs,y)=q_N(y)(\alpha_N(y)-\beta_N(y))+\beta_N(y)<1,$$
and then, it suffices to take $\ell=N$, $(z,y)=y$ and  for each $\varepsilon>0$, we choose:
$$q_N^\varepsilon(y)=q_N(y)+\delta(\varepsilon)\textrm{ with } \delta(\varepsilon)=\varepsilon\wedge\left(\frac{1-q_N(y)}{2}\right).$$
Since $q_N(y)\in(0,1)$, we have that $\delta(\varepsilon)>0$ and $q_N^\varepsilon(y)\in(0,1)$. These choices induce a new probability $\Qs^{\varepsilon}\sim P$
which verifies clearly conditions (1)-$N$,(2)-$N$ and (3)-$N$.\\
Now, we suppose that the assertion is true at the level $k+1$ and we prove that it is also true at the level $k$.\\
If  $\rho_k^+(\Qs,y)<1$, then:
\begin{equation}\label{rhol1}
\rho_k^+(\Qs,y)=q_k(y)\alpha_k(y)\rho_{k+1}^+(\Qs,y\star u)+(1-q_k(y))\beta_k(y)\rho_{k+1}^+(\Qs,y\star d)<1.
\end{equation}
At this point, there are three situations.

(i) If $\alpha_k(y)\rho_{k+1}^+(\Qs,y\star u)>\beta_k(y)\rho_{k+1}^+(\Qs,y\star d)$, we can take $\ell=k$, $(z,y)=y$ and  for each $\varepsilon>0$, we choose:
$$q_k^\varepsilon(y)=q_k(y)+\delta(\varepsilon)\textrm{ with } \delta(\varepsilon)=\varepsilon\wedge\left(\frac{1-q_k(y)}{2}\right),$$
and we can conclude as in the case $k=N$.

(ii) If $\alpha_k(y)\rho_{k+1}^+(\Qs,y\star u)<\beta_k(y)\rho_{k+1}^+(\Qs,y\star d)$, we can take $\ell=k$, $(z,y)=y$ and  for each $\varepsilon>0$, we choose:
$$q_k^\varepsilon(y)=q_k(y)-\delta(\varepsilon)\textrm{ with } \delta(\varepsilon)=\varepsilon\wedge\frac{q_k(y)}{2}.$$
Using similar arguments as before, we achieve the proof in this case.

(iii) If $\alpha_k(y)\rho_{k+1}^+(\Qs,y\star u)=\beta_k(y)\rho_{k+1}^+(\Qs,y\star d)$, then
$\rho_{k+1}^+(\Qs,y\star u)<1$ (because $\alpha_k(y)>\beta_k(y)$). Applying induction hypothesis at the level $k+1$ to $y\star u$, we obtain $\ell\geq k+1$, $(z,y\star u)\in E_{\ell-1}$
and for each $\varepsilon>0$ a probability measure $\Qs^\varepsilon\sim P$ verifying (1)-$(k+1)$, (2)-$(k+1)$ and (3)-$(k+1)$. (1)-$k$ and (2)-$k$ remain the same. Finally, condition (3)-$k$ follows by
plugging condition (3)-$(k+1)$ in \ref{rhol1}. The proof is finished.\\

\end{proof}
Before to establish the characterization theorem for $\Ms(\lambda_c)$, we fix some notations which will be useful in the proof. For each $\Qs\in \Ps_1(\Omega)$,
we put:
$$\As(\Qs)=\left\{(n,x): 1\leq n\leq N,\, x\in E_{n-1}\textrm{ s.t. }\frac{\rho_n^+(\Qs,x)}{\rho_n^-(\Qs,x)}=\rho(\Qs)\right\},$$
and we note that, if $\Qs\in\Ms(\lambda_c)$, then:
$$\As(\Qs)=\left\{(n,x): 1\leq n\leq N,\, x\in A_n(\Qs)\right\}.$$
Until now, we know how to perturb a measure $\Qs\in\Ms(\lambda_c)$ in order to increase the ratio \eqref{ratio} for a point $x\in A_{k_\Qs}(\Qs)$. If in addition
we want to perturb the measure in such a way that the sets $\As(\Qs)$ decrease (in the sense of the inclusion) we need to look to the quantity:

$$\tilde{\rho}(\Qs)=\min_{(n,x)\notin \As(\Qs)}\frac{\rho_n^+(\Qs,x)}{\rho_n^-(\Qs,x)},$$
using by convention $\tilde{\rho}(\Qs)=\rho(\Qs)$ if $\frac{\rho_n^+(\Qs,x)}{\rho_n^-(\Qs,x)}=\rho(\Qs)$ for all $n\in\{1,...,N\}$ and $x\in E_{n-1}$.
We define also:
$$\eta(\Qs)=\tilde{\rho}(\Qs)- \rho(\Qs)\geq 0.$$
\begin{theorem}
 We have that $\Ms(\lambda_c)\neq \emptyset$ if and only if for all $n\in\{1,...,N\}$ and  $x\in E_{n-1}$:
$$\beta_n(x)<1< \alpha_n(x),$$
 and in this case $\Ms(\lambda_c)=\{\Qs^0\}$, where $\Qs^0$ is the probability measure defined in the Paragraph \ref{ssnac}.
\end{theorem}

\begin{proof}
($\Leftarrow$) If $\beta_n(x)<1< \alpha_n(x)$ for all $n\in\{1,...,N\}$ and  $x\in E_{n-1}$, then $\lambda_c=0$ (see Remark \ref{r6}) and the result is a consequence of the
no arbitrage condition in the frictionless case (see Paragraph \ref{ssnac}).\\
($\Rightarrow$) Assume that there exists $n_0\in\{1,...,N\}$ and  $x^*\in E_{n-1}$ such that $\beta_{n_0}(x^*)\geq 1$ or $\alpha_{n_0}(x^*)\leq 1$. We will
prove that $\Ms(\lambda_c)=\emptyset$. In order to do this, we proceed by contradiction, that means we suppose that there exists $\Qs\in\Ms(\lambda_c)$.
Since the no arbitrage condition for the frictionless case is not satisfied, we deduce that $\lambda_c>0$.\\
We set $k=k_\Qs$ and we fix $x\in A_k(\Qs)$. Thanks to Lemma \ref{laux}, we know that either $r_k^+(\Qs,x)>1$ or $r_k^-(\Qs,x)<1$.

(i) If $r_k^+(\Qs,x)>1$, by continuity, we can find $\delta_1>0$ such that:
$$d_\infty(\Qs,\widehat{\Qs})\leq\delta_1\Rightarrow r_k^+(\widehat{\Qs},x)>1.$$
(i.1) If $\eta=\eta(\Qs)>0$, again by continuity, we can find $\delta_2>0$ such that:
$$d_\infty(\Qs,\widehat{\Qs})\leq\delta_2\Rightarrow \max_{n\in\{1,...,N\}}\left\{\max_{y\in E_{n-1}}\left|\frac{\rho_n^+(\Qs,y)}{\rho_n^-(\Qs,y)}-\frac{\rho_n^+(\widehat{\Qs},y)}{\rho_n^-(\widehat{\Qs},y)}\right|\right\}\leq\frac{\eta}{2}.$$
Since $\rho_k^-(\Qs,x)>1$, by using Lemma \ref{lmon}, we can associate to $\varepsilon=\delta_1\wedge\delta_2>0$ a probability
measure $\Qs^\varepsilon$ verifying \eqref{lc1}, \eqref{lc2} and \eqref{lc3}. Conditions \eqref{lc1} and \eqref{lc2} and the fact that $\varepsilon\leq \delta_1$ implies that:
$$\rho_k^+(\Qs^\varepsilon,x)=\rho_k^+(\Qs,x)=1.$$
Using this and \eqref{lc3}, we obtain that:
$$\frac{\rho_k^+(\Qs^\varepsilon,x)}{\rho_k^-(\Qs^\varepsilon,x)}=\frac{1}{\rho_k^-(\Qs^\varepsilon,x)}>\frac{1}{\rho_k^-(\Qs,x)}=\frac{\rho_k^+(\Qs,x)}{\rho_k^-(\Qs,x)}=1-\lambda_c$$
and for each $m\leq k$ and $y\in E_{m-1}$:
\begin{equation}\label{eaux1}
\frac{\rho_m^+(\Qs^\varepsilon,y)}{\rho_m^-(\Qs^\varepsilon,y)}\geq \frac{\rho_m^+(\Qs,y)}{\rho_m^-(\Qs,y)}
\end{equation}
For the last assertion is crucial that, passing from $\Qs$ to $\Qs^\varepsilon$, the only change at the level $k$ is in the quantity $\rho_k^-(\Qs,x)$
which decreases.\\
Now, for each $n\in\{1,...,N\}$ and $y\notin A_n(\Qs)$, since $\varepsilon\leq\delta_2$, we have:
$$\frac{\rho_n^+(\Qs^\varepsilon,y)}{\rho_n^-(\Qs^\varepsilon,y)}\geq -\frac{\eta}{2}+\frac{\rho_n^+(\Qs,y)}{\rho_n^-(\Qs,y)}\geq -\frac{\eta}{2}+\tilde{\rho}(\Qs)=\frac{\eta}{2}+ 1-\lambda_c>1-\lambda_c.$$
From this and \eqref{eaux1},  we can deduce that $\rho(\Qs^\varepsilon)=1-\lambda_c$ and that:
$$A_n(\Qs^\varepsilon)\subseteq A_n(\Qs),$$
for each $n\in\{1,...,N\}$, the inclusion being strict for $n=k$.\\
(i.2) If $\eta(\Qs)=0$, we proceed in the same way, but taking $\varepsilon=\delta_1$. The arguments remain the same until \eqref{eaux1}
and from there, using that in this case $k=N$, we can obtain the same conclusion.

(ii) If $r_k^-(\Qs,x)<1$, by continuity, we can find $\delta_3>0$ such that:
$$d_\infty(\Qs,\widehat{\Qs})\leq\delta_3\Rightarrow r_k^-(\widehat{\Qs},x)<1$$
and the arguments are similar to those in (i), but taking now $\varepsilon=\delta_2\wedge\delta_3$ when $\eta(\Qs)>0$ and $\varepsilon=\delta_3$ in the other case. The only difference, is that now the probability measure
$\Qs^\varepsilon$ verifies:
$$\rho_k^-(\Qs^\varepsilon,x)=\rho_k^-(\Qs,x)=1\textrm{ and }\rho_k^+(\Qs^\varepsilon,x)\geq\rho_k^+(\Qs,x),$$
but the conclusions are the same.

Summarizing, starting from $\Qs\sim P$ satisfying $\rho(\Qs)=1-\lambda_c$, we construct a probability measure $\Qs^{(1)}=\Qs^\varepsilon\sim P$ such that $\rho(\Qs^{(1)})=1-\lambda_c$ and $\nu(\Qs^{(1)})<\nu(\Qs)$. We repeat the procedure
inductively, starting at each time with a probability measure $\Qs^{(i)}\sim P$ satisfying $\rho(\Qs^{(i)})=1-\lambda_c$ and constructing a new probability measure $\Qs^{(i+1)}\sim P$ such that $\rho(\Qs^{(i+1)})=1-\lambda_c$ and
$\nu(\Qs^{(i+1)})<\nu(\Qs^{(i)})$.  Necessarily, at some point we will arrive to a probability measure $\Qs^{(n_0)}\sim P$ verifying $\rho(\Qs^{(n_0)})=1-\lambda_c$ and $\nu(\Qs^{(n_0)})=0$, which is a contradiction.
\end{proof}
\section{Homogeneous and semi-homogeneous binary markets}
In this section, we are interested to deduce, from the characterization of $\lambda_c$ (Theorem \ref{thmc} or Corollary \ref{lac}), more explicit expressions
in some special cases of binary markets. More precisely, we focus in the two following cases:

\begin{itemize}
 \item \textit{Homogeneous case}: We refer to this case, when the parameters of the model are homogeneous in time and space. That means:
 $$0<\beta_n(x)=\beta<\alpha_n(x)=\alpha,$$
 for all $n\in\{1,...,N\}$ and $x\in E_{n-1}$.
 \item \textit{Semi-homogeneous case}: We refer to this case, when the parameters of the model are not necessarily homogeneous in time, but they are still homogeneous in space. That means:
 $$0<\beta_n(x)=\beta_n<\alpha_n(x)=\alpha_n,$$
for all $n\in\{1,...,N\}$ and $x\in E_{n-1}$.
 \end{itemize}

Henceforth, we assume that our binary markets are semi--homogeneous (the homogeneous case is covered). In this framework, we start by giving an upper bound for $\lambda_c$ and then we will prove that
this upper bound coincides with $\lambda_c$ for homogeneous binary markets and also for a large class of semi-homogeneous binary markets. In order to do this, based on the characterization of
$\lambda_c$ given by Corollary \ref{lac}, we construct a probability measure, by taking at each time the best  ``1-step'' choice, which gives us, by means of $\rho$, a natural upper bound and also a naive
candidate for the critical transaction cost $\lambda_c$.

Let $\Qs^*$ be the probability measure defined by:
$$q_n(\Qs^*,x)=:q_n^*=1_{\{\alpha_n\leq 1\}}+\frac{1-\beta_n}{\alpha_n-\beta_n}1_{\{\beta_n<1<\alpha_n\}},\quad n\in\{1,...,N\},\, x\in E_{n-1}.$$
and define the sequences of positive numbers ${\{\varrho_n^+\}}_{n=1}^{N+1}$, ${\{\varrho_n^-\}}_{n=1}^{N+1}$ and ${\{\gamma_n\}}_{n=1}^N$ by setting:
$$\varrho_{N+1}^+=\varrho_{N+1}^-=1$$
and for $n \in\{1,...,N\}$:
$$\gamma_n= \alpha_n 1_{\{\alpha_n\leq 1\}}+\beta_n 1_{\{\beta_n\geq 1\}}+1_{\{\beta_n<1<\alpha_n\}},$$
$$\varrho_n^+=1\wedge\left[\gamma_n\,\varrho_{n+1}^+\right]\quad\textrm{and}\quad\varrho_n^-=1\vee\left[\gamma_n\,\varrho_{n+1}^-\right].$$
The relation between these sequences of numbers and the functions $\rho^+$ and $\rho^-$ is given in the following lemma:
\begin{lemma}\label{lid}
For all $n\in\{1,...,N+1\}$ and $x\in E_{n-1}$:
$$\varrho_n^+= \rho_n^+(\Qs^*,x) \quad\textrm{and}\quad\varrho_n^-=\rho_n^-(\Qs^*,x)$$
\end{lemma}
\begin{proof}
We prove this by using a backward recurrence. By definition, the result is true for $n=N+1$. Now, we assume the result holds for $n+1$ and we prove it for $n$.
By definition of $\rho^+$ and the recurrence step, we obtain that:
$$\rho_{n}^+(\Qs^*,x)=1\wedge\left[\left(\,q_{n}^*\,\alpha_{n}+(1-q_{n}^*)\,\beta_{n}\right)\varrho_{n+1}^+\right],$$
and now, using the definitions of $q_n^*$ and $\gamma_n$:
$$\rho_{n}^+(\Qs^*,x)=1\wedge\left[\gamma_n\,\varrho_{n+1}^+\right].$$
The result follows from the definition of $\rho_n^+$. The proof for $\rho^-$ is analogous.
\end{proof}

In order to obtain an easy expression for $\varrho_n^+$ and $\varrho_n^-$, we introduce the sets:
$$\Lambda_n=\{1\}\cup\left\{\prod\limits_{\ell=0}^{k}\gamma_{n+\ell}:\, 0\leq k\leq N-n\right\}.$$

\begin{lemma}\label{fcr}
For all $n\in\{1,...,N\}$:
$$\varrho_n^+= \min \Lambda_n \quad\textrm{and}\quad\varrho_n^-= \max \Lambda_n.$$
\end{lemma}
\begin{proof}
The result follows from a simple backward recurrence and the fact that:
$$\Lambda_n=\{1\}\cup \gamma_n\Lambda_{n+1}.$$
\end{proof}

\begin{remark}\label{rme}
If we define the sets:
$$\Lambda_n^*=\left\{\frac{x}{y}:\,x,y\in \Lambda_n\right\}\quad\textrm{and}\quad\Lambda_n^0=\left\{\prod\limits_{\ell=p}^k\gamma_{n+\ell}:\, 0\leq p\leq k\leq N-n\right\},$$
we can see that:
$$\frac{\min \Lambda_n}{\max \Lambda_n}=\min \Lambda_n^*=1\wedge\min\Lambda_n^0\wedge\frac{1}{\max\Lambda_n^0}.$$
\end{remark}

\begin{proposition}[Upper bound in the semi-homogeneous case]\label{ubsh} We have that:
$$\lambda_c\leq 1-1\,\wedge\,\min\Lambda_1^0\,\wedge\,\frac{1}{\max \Lambda_1^0}.$$
\end{proposition}
\begin{proof}
From Corollary \ref{lac} and using Lemmas \ref{lid} and \ref{fcr}, we obtain that:
\begin{equation}
\lambda_c\leq 1- \min\limits_{n\in\{1,...,N\}}\left\{\frac{\min \Lambda_n}{\max \Lambda_n}\right\}.
\end{equation}
The statement follows from this inequality, Remark \ref{rme} and the fact that:
$$\Lambda_N^0\subseteq\Lambda_{N-1}^0\subseteq\cdots\subseteq\Lambda_1^0.$$
\end{proof}
The next Proposition covers the homogeneous case, as well as a large class of semi-homogeneous cases.
\begin{proposition}\label{pec}In the semi-homogeneous case:
\begin{enumerate}
\item If $\alpha_n\leq 1$, for all $n\in\{1,...,N\}$, then:
$$\lambda_c=1-\prod_{n=1}^N\alpha_n.$$
\item If $\beta_n\geq 1$, for all $n\in\{1,...,N\}$, then:
$$\lambda_c=1-\prod_{n=1}^N\frac1{\beta_n}.$$
\item If $\beta_n\leq 1\leq \alpha_n$, for all $n\in\{1,...,N\}$, then:
$$\lambda_c=0.$$
\end{enumerate}
\end{proposition}
\begin{proof}
(1) We fix $\Qs\sim P$ and we will prove by a backward recurrence that for each $k\in\{1,...,N\}$:
$$\rho_k^+(\Qs,x)\leq \prod_{n=k}^N\alpha_n,$$
for all $x\in E_{k-1}$. In fact, for $k=N$, the statement is true by definition. Now, we suppose that:
$$\rho_{k+1}^+(\Qs,y)\leq\prod_{n=k+1}^N\alpha_n, $$
for all $y\in E_k$. By using this and the definition of $\rho_k^+$, we have that for each $x\in E_{k-1}$:
$$\rho_{k}^+(\Qs,x)\leq \alpha_k \prod_{n=k+1}^N\alpha_n.$$
This proves our statement. We can deduce that:
$$\rho(\Qs)\leq \prod_{n=1}^N\alpha_n.$$
As $\Qs$ is arbitrary, we obtain that: $\lambda_c\geq 1-\prod\limits_{n=1}^N\alpha_n$.

On the other hand, it follows from the definitions and Lemma \ref{fcr} that:
 $$\gamma_n=\alpha_n,\,\varrho_n^+=\prod_{k=n}^N\alpha_k\quad\textrm{and}\quad \varrho_n^-=1,$$
and the result is a consequence  of Lemma \ref{fcr}, Remark \ref{rme} and Proposition \ref{ubsh}.

(2) The proof of this case uses the same argument like the previous one.

(3) We have proved that in Remark \ref{r6}. However, we provide here a proof using the results of this section.
Note that by definition:
$$\gamma_n=1,\,\varrho_n^+=\varrho_n^-=1,$$
and then, from Lemma \ref{fcr}, Remark \ref{rme} and Proposition \ref{ubsh}, the result follows.
\end{proof}

\bibliographystyle{plain}
\bibliography{reference}

\end{document}